\documentclass[11pt,reqno,a4paper]{amsart}
\usepackage{euscript}

\newtheorem{theorem}{Theorem}
\newtheorem{proposition}[theorem]{Proposition}
\newtheorem{lemma}{Lemma}
\newtheorem{example}{Example}
\newtheorem{definition}{Definition}

\renewcommand{\epsilon}{\varepsilon}

\def\Id{\text{\rm Id}}
\def\cA{\EuScript{A}}
\def\N{\mathbb{N}}
\def\Z{\mathbb{Z}}
\def\R{\mathbb{R}}

\begin{document}

\title{Admissibility and nonuniformly hyperbolic sets}

\begin{abstract}
We obtain a characterization of two classes of dynamics with nonuniformly hyperbolic behavior in terms of an admissibility property. Namely, we consider exponential dichotomies with respect to a sequence of norms and nonuniformly hyperbolic sets. We note that the approach to establishing exponential bounds along the stable and the unstable directions differs from the standard technique of substituting test sequences. Moreover, we obtain the bounds in a single step.
\end{abstract}

\begin{thanks}
{L.B. and C.V. were supported by Portuguese National Funds through FCT - Funda\c{c}\~ao para a Ci\^encia e a Tecnologia within project PTDC/MAT/117106/2010 and by CAMGSD}
\end{thanks}

\author{Luis Barreira}
\address{Departamento de Matem\'atica,
Instituto Superior T\'ecnico, 1049-001 Lisboa, Portugal}
\email{barreira@math.ist.utl.pt}

\author{Davor Dragi\v cevi\'c}
\address{Department of Mathematics,
University of Rijeka, 51000 Rijeka, Croatia}
\email{ddragicevic@math.uniri.hr}

\author{Claudia Valls}
\address{Departamento de Matem\'atica,
Instituto Superior T\'ecnico, 1049-001 Lisboa, Portugal}
\email{cvalls@math.ist.utl.pt}


\maketitle

\section{Introduction}

Our main objective is to obtain a characterization of two classes of dynamics with nonuniformly hyperbolic behavior in terms of an admissibility property. Namely, we consider the class of exponential dichotomies with respect to a sequence of norms and the class of nonuniformly hyperbolic sets.

In the first part of the paper we consider a nonautonomous dynamics with discrete time obtained from a sequence of linear operators on a Banach space and we characterize the notion of an exponential dichotomy with respect to a sequence of norms. The principal motivation for considering this notion is that includes both the notions of a uniform and of a nonuniform exponential dichotomy as special cases.
We refer the reader to the books \cite{CLb, H, He, SY} for details and further references on the uniform theory. On the other hand, the requirement of uniformity for the asymptotic behavior is often too stringent for the dynamics and it turns out that the notion of a nonuniform exponential dichotomy is much more typical. We refer the reader to \cite{BP2} for an account of a substantial part of the theory.
Most of the work in the literature related to admissibility has been devoted to the study of uniform exponential dichotomies. For some of the most relevant early contributions in the area we refer to the books by Massera and Sch\"affer~\cite{MS} and by Dalec$'$ki\u{\i} and Kre{\u\i}n \cite{DK}. We also refer to~\cite{LV} for some early results in infinite-dimensional spaces. For a detailed list of references, we refer the reader to \cite{CLb} and for more recent work to Huy~\cite{HH}.

We emphasize that we consider the general case of a noninvertible dynamics which means that we assume only the invertibility along the unstable direction. Moreover, we characterize exponential dichotomies with respect to a sequence of norms in terms of the admissibility of a large family of Banach spaces (the particular case of $l^p$ spaces was considered in~\cite{BDV}). We note that the approach to establishing exponential bounds along the stable and the unstable directions differs from the standard technique of substituting test sequences (see for example~\cite{H, HH}). Moreover, in contrast to the existing approaches, we are able to obtain bounds along the stable and unstable directions in a single step.

In the second part of the paper we obtain an analogous characterization of nonuniformly hyperbolic sets. The notion of a nonuniformly hyperbolic set arises naturally in the context of smooth ergodic theory. Indeed, if $f$ is a $C^1$ diffeomorphism of a finite-dimensional compact manifold preserving a finite measure $\mu$ with nonzero Lyapunov exponents, then there exists a nonuniformly hyperbolic set of full $\mu$-measure. We refer the reader to \cite{BP2} for details.
Our work is close in spirit to that of Mather~\cite{M}, who obtained a similar characterization of uniformly hyperbolic sets, as well as that of Dragi\v{c}evi\'c and Slijep\v{c}evi\'c~\cite{DS}, where the problem of extending Mather's result to nonuniformly hyperbolic dynamics was first considered. However, there are substantial differences between our approach and that in~\cite{DS}, which provides a characterization of ergodic invariant measures with nonzero Lyapunov exponents and not of nonuniformly hyperbolic sets.

\section{Preliminaries}

In this section we introduce a few basic notions. Let $\mathcal{S}$ be the set of all sequences $\mathbf{s}=(s_n)_{n\in\Z}$ of real numbers. We say that a linear subspace $B\subset \mathcal{S}$ is a \emph{normed sequence space} if there exists a norm $\lVert \cdot \rVert_B \colon B \to \R_0^+$ such that if $\mathbf{s}'\in B$ and $\lvert s_n\rvert \le \lvert s_n'\rvert$ for $n\in \Z$, then $\mathbf{s}\in B$ and $\lVert \mathbf{s}\rVert_B \le \lVert \mathbf{s}'\rVert_B$. If in addition $(B, \lVert \cdot \rVert_B)$ is complete, we say that $B$ is a \emph{Banach sequence space}.

Let $B$ be a Banach sequence space. We say that $B$ is \emph{admissible} if:
\begin{enumerate}
\item
$\chi_{\{n\}} \in B$ and $\lVert \chi_{\{n\}}\rVert_B >0$ for $n\in \Z$, where $\chi_A$ denotes the characteristic function of the set $A\subset \Z$;
\item
for each $\mathbf{s}=(s_n)_{n\in \Z}\in B$ and $m\in \Z$, the sequence $\mathbf{s}^m=(s_n^m)_{n\in \Z}$ defined by $s_n^m=s_{n+m}$ belongs to $B$ and there exists $N>0$ such that $\lVert \mathbf{s}^m \rVert_B \le N \lVert \mathbf{s}\rVert_B$ for $\mathbf{s}\in B$ and $m\in \Z$.
\end{enumerate}

We present some examples of Banach sequence spaces.

\begin{example}
The set $l^\infty =\{ \mathbf{s} \in \mathcal{S} : \sup_{n\in \Z} \lvert s_n \rvert < +\infty \}$ is a Banach sequence space when equipped with the norm $\lVert \mathbf{s} \rVert =\sup_{n\in \Z} \lvert s_n \rvert$.
\end{example}

\begin{example}
For each $p\in [1, \infty )$, the set $l^p= \{ \mathbf{s} \in \mathcal{S}: \sum_{n\in \Z} \lvert s_n \rvert^p <+\infty \}$ is a Banach sequence space when equipped with the norm $\lVert \mathbf{s} \rVert= (\sum_{n\in \Z} \lvert s_n \rvert^p )^{1/p}$.
\end{example}

\begin{example}
Let $\phi \colon (0, +\infty) \to (0, +\infty]$ be a nondecreasing nonconstant left-continuous function. We set $\psi(t)=\int_0^t \phi(s) \, ds$  for $t\ge 0$. Moreover, for each $\mathbf s \in \mathcal S$, let $M_\phi (\mathbf s)=\sum_{n\in \Z} \psi(\lvert s_n \rvert)$. Then
\[
B=\bigl\{ \mathbf s \in \mathcal S : M_\phi (c \mathbf s)<+\infty \ \text{for some} \ c>0 \bigr\}
\]
is a Banach sequence space when equipped with the norm
\[
\lVert \mathbf s \rVert=\inf \bigl\{ c>0 : M_\phi ( \mathbf s / c ) \le 1 \bigr\}.
\]
\end{example}

We need the following auxiliary results.

\begin{proposition}\label{l2}
Let $B$ be an admissible Banach sequence space.
\begin{enumerate}
\item
If $\mathbf{s}^1=(s_n^1)_{n\in \Z}$ and $\mathbf{s}^2=(s_n^2)_{n\in \Z}$ are sequences in $\mathcal{S}$ and $s_n^1=s_n^2$ for all but finitely many $n\in \Z$, then $\mathbf{s}^1 \in B$ if and only if $\mathbf{s}^2 \in B$.
\item
If $\mathbf{s}^n \to \mathbf{s}$ in $B$ when $n \to \infty$, then $s_m^n \to s_m$ when $n \to \infty$, for $m\in \Z$.
\item
For each $\mathbf{s}\in B$ and $\lambda \in (0,1)$, the sequences $\mathbf{s}^1$ and $\mathbf{s}^2$ defined by
\[
s_n^1=\sum_{m\ge 0} \lambda^m s_{n-m} \quad \text{and} \quad s_n^2=\sum_{m\ge 1} \lambda^m s_{n+m}
\]
are in $B$, and
\begin{equation}\label{bnd}
\lVert \mathbf{s}^1 \rVert_B \le \frac{N}{1-\lambda} \lVert \mathbf{s} \rVert_B \quad \text{and} \quad \lVert \mathbf{s}^2 \rVert_B \le \frac{N\lambda}{1-\lambda} \lVert \mathbf{s} \rVert_B.
\end{equation}
\end{enumerate}
\end{proposition}

\begin{proof}
1. Assume that $\mathbf{s}^1 \in B$ and let $I\subset \Z$ be the finite set of all integers $n\in \Z$ such that $s_n^1 \neq s_n^2$. We define $\mathbf{v}=(v_n)_{n\in \Z}$ by $v_n=0$ if $n\notin I$ and $v_n=s_n^2-s_n^1$ if $n\in I$.
Since $B$ is an admissible Banach sequence space, we have $\mathbf{v}\in B$ and thus $\mathbf{s}^2=\mathbf{s}^1+\mathbf{v} \in B$.

2. We have
\[
\lvert s_m^n-s_m\rvert \chi_{\{m\}}(k) \le \lvert s_k^n-s_k\rvert
\]
for $k\in \Z$ and $n\in \N$. By the definition of a normed sequence space, we obtain
\[
\lvert s_m^n-s_m\rvert \le \frac{N}{\lVert \chi_{\{0\}}\rVert_B}\lVert s^n-s\rVert_B
\]
for $n\in \Z$ and the conclusion follows.

3. We define a sequence $\mathbf{v}=(v_n)_{n\in \Z}$ by $v_n=\lvert s_n \rvert$ for $n\in \Z$. Clearly, $\mathbf{v}\in B$ and $\lVert \mathbf{v}\rVert_B=\lVert \mathbf{s}\rVert_B$. Moreover,
\[
\sum_{m\ge 0} \lambda^m \lVert \mathbf{v}^{-m}\rVert_B \le N\sum_{m\ge 0}\lambda^m\lVert \mathbf{v}\rVert_B=\frac{N}{1-\lambda}\lVert \mathbf{s}\rVert_B< +\infty.
\]
Since $B$ is complete, the series $\sum_{m\ge 0} \lambda^m \mathbf{v}^{-m}$ converges to some sequence $\mathbf{x}=(x_n)_{n\in \Z} \in B$. It follows from the second property that
\[
x_n=\sum_{m\ge 0}\lambda^m\lvert s_{n-m}\rvert
\]
for $n\in \Z$. Since $\lvert s_n^1\rvert \le \lvert x_n \rvert$ for $n \in \Z$, we conclude that $\mathbf{s}^1 \in B$ and $\lVert \mathbf{s}^1\rVert_B \le \lVert \mathbf{x}\rVert_B$, which yields that the first inequality in~\eqref{bnd} holds. One can show in a similar manner that $\mathbf{s}^2 \in B$ and that the second inequality in~\eqref{bnd} holds.
\end{proof}

Now let $(X, \lVert \cdot \rVert)$ be a Banach space and let $\lVert \cdot \rVert_n$, for $n\in \Z$, be a sequence of norms on $X$ such that $\lVert \cdot \rVert_n$ is equivalent to $\lVert \cdot \rVert$ for each $n \in \Z$. For an admissible space $B$, let
\[
Y_B=\big\{\mathbf{x}=(x_n)_{n\in \Z} \subset X: (\lVert x_n \rVert_n)_{n\in \Z} \in B\big\}.
\]
For $\mathbf{x}\in Y_B$, we define
\[
\lVert \mathbf{x}\rVert_{Y_B} =\lVert (\lVert x_n \rVert_n)_{n\in \Z}\rVert_B.
\]

\begin{proposition}
$(Y_B, \lVert \cdot \rVert_{Y_B})$ is a Banach space.
\end{proposition}

\begin{proof}
Let $(\mathbf{x}^k)_{k\in \N}$ be a Cauchy sequence in $Y_B$. Repeating arguments in the proof of Proposition~\ref{l2}, one can show that $(x_n^k)_{k\in \N}$ is a Cauchy sequence in $X$ for each $n\in \Z$.
Let
\[
x_n=\lim_{k \to \infty} x_n^k \quad \text{for} \quad n\in \Z
\]
and let $\mathbf{s}^k=(\lVert x_n^k\rVert_n)_{n\in \Z} \in B$ for $k\in \N$. Since
\[
\big\lvert \lVert x_n^k\rVert_n -\lVert x_n^l\rVert_n \big\rvert \le \lVert x_n^k-x_n^l\rVert_n \quad \text{for} \quad n \in \Z,
\]
we conclude that
\[
\lVert \mathbf{s}^k-\mathbf{s}^l \rVert_B \le \lVert \mathbf{x}^k-\mathbf{x}^l\rVert_{Y_B}\quad \text{for} \quad k, l\in \N.
\]
Hence,
$(\mathbf{s}^k)_{k\in \N}$ is a Cauchy sequence in $B$. Since $B$ is complete, it follows from property~2 in Proposition~\ref{l2} that $\mathbf{s}^k \to \mathbf s$ in $B$ when $k \to \infty$, where $s_n=\lVert x_n\rVert_n$ for $n\in \Z$. In particular, $\mathbf x=(x_n)_{n\in \Z} \in Y_B$. One can easily verify that the sequence $(\mathbf x^k-\mathbf x)_{k\in \N}$ converges to $0$ in $Y_B$, which implies that $(\mathbf{x}^k)_{k\in \N}$ converges to $\mathbf x$ in~$Y_B$.
\end{proof}

\section{Admissibility and exponential dichotomies}\label{S3}

In this section we consider the notion of an exponential dichotomy with respect to a sequence of norms and we characterize it in terms of the invertibility of a certain linear operator.

\subsection{Basic notions}

Let $X$ be a Banach space and let $L(X)$ be the set of all bounded linear operators from $X$ to itself.
Given a sequence $(A_m)_{m\in \Z}$ in $B(X)$, let
\begin{equation}\label{*AC}
\cA(n,m)=\begin{cases}
A_{n-1} \cdots A_m & \text{if $n > m$,} \\
\Id & \text{if $n=m$}.
\end{cases}
\end{equation}

\begin{definition}\label{DEF}
We say that $(A_m)_{m\in \Z}$ admits an \emph{exponential dichotomy} with respect to the sequence of norms $\lVert \cdot \rVert_m $ if:
\begin{enumerate}
\item
there exist projections $P_m\colon X\to X$ for each $m\in \Z$ satisfying
\begin{equation}\label{*AX}
A_m P_m=P_{m+1}A_m \quad \text{for} \quad m\in \Z
\end{equation}
such that each map $A_m \rvert \ker P_m \colon \ker P_m \to \ker P_{m+1}$ is invertible;
\item
there exist constants $D>0$ and $0<\lambda<1<\mu$ such that for each $x \in X$ and $n,m \in \Z$ we have
\begin{equation}\label{eq:1}
\lVert \cA(n,m)P_m x\rVert_n \le D \lambda^{n-m} \lVert x\rVert_m \quad \text{for} \quad n\ge m
\end{equation}
and
\begin{equation}\label{eq:2}
\lVert \cA(n,m)Q_m x\rVert_n \le D \mu^{n-m} \lVert x\rVert_m\quad \text{for} \quad n\le m,
\end{equation}
where $Q_m=\Id-P_m$ and
\[
\cA(n,m)=\left(\cA(m,n)\rvert \ker P_n\right)^{-1} \colon \ker P_m \to \ker P_n
\]
for $n<m$.
\end{enumerate}
\end{definition}

More generally, one can consider the notion of an exponential dichotomy for sequences of linear operators between different spaces. Namely, let $X_n=(X_n,\lVert \cdot \rVert)$, for $n \in \Z$, be pairwise isomorphic Banach spaces. Given a sequence of bounded linear operators $A_m \colon X_m \to X_{m+1}$, for $m \in \Z$, one can define $\cA(n,m) \colon X_m \to X_n$ by~\eqref{*AC} and introduce a corresponding notion of an exponential dichotomy, with projections $P_m \colon X_m \to X_m$ for $m \in \Z$. All the results obtained in this section hold verbatim in this general setting, but we prefer avoiding the cumbersome notation.

Now let $B$ be a Banach sequence space. Our main aim is to characterize the notion of an exponential dichotomy with respect to a sequence of norms in terms of the invertibility of the operator $T_B \colon \mathcal{D}(T_B) \subset Y_B \to Y_B$ defined by
\[
(T_B \mathbf{x})_n=x_n-A_{n-1}x_{n-1}, \quad n\in \Z,
\]
on the domain $\mathcal{D}(T_B)$ formed by all vectors $\mathbf{x}\in Y_B$ such that $T_B \mathbf{x}\in Y_B$.

\begin{proposition}\label{propo.1}
The linear operator $T_B \colon \mathcal{D}(T_B) \subset Y_B \to Y_B$ is closed.
\end{proposition}

\begin{proof}
Let $(\mathbf{x}^k)_{k\in \N}$ be a sequence in $\mathcal{D}(T_B)$ converging to $\mathbf{x}\in Y_B$ such that $T_B\mathbf{x}^k$ converges to $\mathbf{y}\in Y_B$. It follows from the definition of $Y_B$ and property~2 in Proposition~\ref{l2} that
\[
x_n-A_{n-1}x_{n-1}=\lim_{k\to \infty}(x_n^k-A_{n-1}x_{n-1}^k)=\lim_{k\to \infty}(T_B\mathbf{x}^k)_n=y_n
\]
for $n\in \Z$, using the continuity of the linear operator $A_{n-1}$. Therefore, $\mathbf{x}\in \mathcal{D}(T_B)$ and $T_B\mathbf{x}=\mathbf{y}$. This shows that the operator $T_B$ is closed.
\end{proof}

For $\mathbf{x}\in \mathcal{D}(T_B)$ we consider the graph norm
\[
\lVert \mathbf{x}\rVert_{Y_B}'=\lVert \mathbf{x}\rVert_{Y_B}+\lVert T\mathbf{x}\rVert_{Y_B}.
\]
Clearly, the operator
\[
T_B\colon (\mathcal{D}(T_B),\lVert \cdot \rVert_{Y_B}') \to (Y,\lVert \cdot \rVert_{Y_B})
\]
is bounded and from now on we denote it simply by~$T_B$. It follows from Proposition~\ref{propo.1} that $(\mathcal{D}(T_B),\lVert \cdot \rVert'_{Y_B})$ is a Banach space.

\subsection{Characterization of exponential dichotomies}

In this section we characterize the notion of an exponential dichotomy with respect to a sequence of norms in terms of the invertibility of the operator $T_B$.

\begin{theorem}\label{thm.1}
If the sequence $(A_m)_{m\in \Z}$ admits an exponential dichotomy with respect to the sequence of norms~$\lVert \cdot \rVert_m$, then the operator $T_B$ is invertible.
\end{theorem}

\begin{proof}
In order to establish the injectivity of the operator $T_B$, assume that $T_B\mathbf{x}=0$ for some $\mathbf{x}\in Y_B$. Then $x_n=A_{n-1}x_{n-1}$ for $n\in \Z$. Let $x_n^s=P_n x_n$ and $x_n^u=Q_n x_n$. We have $x_n=x_n^s+x_n^u$ and it follows from~\eqref{*AX} that
\[
x_n^s=A_{n-1}x_{n-1}^s \quad \text{and} \quad
x_n^u=A_{n-1}x_{n-1}^u
\]
for $n\in \Z$. Moreover, $x_k^s= \cA(k,k-m)x_{k-m}^s$ for $m\ge 0$ and hence,
\[
\begin{split}
\lVert x_k^s\rVert_k &=\lVert \cA(k,k-m)x_{k-m}^s\rVert_k \\
&=\lVert \cA(k,k-m)P_{k-m}x_{k-m}\rVert_k \\
&\le D \lambda^m \lVert x_{k-m}\rVert_{k-m} \\
&\le \frac{DN}{\alpha_B} \lambda^{m}\lVert \mathbf{x}\rVert_{Y_B},
\end{split}
\]
where $\alpha_B=\lVert \chi_{\{0\}} \rVert_B$. Letting $m \to \infty$ in the last term yields that $x_k^s=0$ for $k\in \Z$.
Similarly, $x_k^u=\cA(k,k+m)x_{k+m}^u$ for $m\ge 0$ and hence,
\[
\begin{split}
\lVert x_k^u\rVert_k &= \lVert \cA(k,k+m)x_{k+m}^u\rVert_k \\
&=\lVert \cA(k,k+m)Q_{k+m}x_{k+m}\rVert_k \\
&\le D\mu^{-m}\lVert x_{k+m}\rVert_{k+m} \\
&\le \frac{DN}{\alpha_B}\mu^{-m}\lVert \mathbf{x}\rVert_{Y_B}.
\end{split}
\]
Therefore, $x_k^u=0$ for $k\in \Z$ and hence $\mathbf{x}=0$. This shows that the operator $T_B$ is injective.

Now we show that $T_B$ is onto. Take $\mathbf{y}=(y_n)_{n\in \Z}\in Y_B$. For each $n\in \Z$, let
\[
x_n^1=\sum_{m\ge 0}\cA(n,n-m) P_{n-m} y_{n-m}
\]
and
\[
x_n^2=-\sum_{m\ge 1}\cA(n,n+m) Q_{n+m} y_{n+m}.
\]
We have
\[
\lVert x_n^1\rVert_n \le \sum_{m\ge 0}D\lambda^{m}\lVert y_{n-m}\rVert_{n-m} \quad \text{and} \quad \lVert x_n^2\rVert_n \le \sum_{m\ge 1}D\mu^{-m}\lVert y_{n+m}\rVert_{n+m}
\]
It follows from property~3 in Proposition~\ref{l2} that $(x_n^1)_{n\in \Z}$ and $(x_n^2)_{n\in \Z}$ belong to $Y_B$. Now let $x_n=x_n^1+x_n^2$ for $n \in \Z$ and $\mathbf{x}=(x_n)_{n\in \Z}$. Then $\mathbf{x}\in Y_B$ and one can easily verify that $T_B \mathbf{x}=\mathbf{y}$. This completes the proof of the theorem.
\end{proof}

Now we establish the converse of Theorem~\ref{thm.1}.

\begin{theorem}
If the operator $T_B$ is bijective, then the sequence $(A_m)_{m\in\Z}$ admits an exponential dichotomy with respect to the sequence of norms $\lVert \cdot \rVert_m$.
\end{theorem}

\begin{proof}
For each $n\in \Z$, let $X(n)$ be the set of all $x\in X$ with the property that there exists a sequence $\mathbf{x}=(x_m)_{m\in \Z}\in Y_B$ such that $x_n=x$ and $x_m=A_{m-1}x_{m-1}$ for $m>n$. Moreover, let $Z(n)$ be the set of all $x\in X$ for which there exists $\mathbf{z}=(z_m)_{m\in \Z}\in Y_B$ such that $z_n=x$ and $z_m=A_{m-1}z_{m-1}$ for $m\le n$.
One can easily verify that $X(n)$ and $Z(n)$ are subspaces of $X$.

\begin{lemma}\label{L}
For each $n\in \Z$, we have
\begin{equation}\label{DEC}
X=X(n)\oplus Z(n).
\end{equation}
\end{lemma}

\begin{proof}[Proof of the lemma]
Given $v\in X$, we define a sequence $\mathbf{y}=(y_m)_{m\in \Z}$ by $y_n=v$ and $y_m=0$ for $m\ne n$.
Clearly, $\mathbf{y}\in Y_B$. Hence, there exists $\mathbf{x}\in Y_B$ such that $T_B\mathbf{x}=\mathbf{y}$, that is,
\begin{equation}\label{a}
x_n-A_{n-1}x_{n-1}=v
\end{equation}
and
\begin{equation}\label{a1}
x_{m+1}=A_mx_m \quad \text{for} \quad m\neq n-1.
\end{equation}
Since $\mathbf{x}\in Y_B$, we obtain
\[
x_n\in X(n) \quad \text{and} \quad A_{n-1}x_{n-1}\in Z(n).
\]
Moreover, by~\eqref{a}, we have $v\in X(n)+ Z(n)$.

Now take $v\in X(n) \cap Z(n)$ and choose $\mathbf{x}=(x_m)_{m\in \Z}$ and $\mathbf{z}=(z_m)_{m\in \Z}$ in $Y_B$ such that $x_n=z_n=v$, \[
x_m=A_{m-1}x_{m-1} \quad \text{for} \quad m>n
\]
and
\[
z_m=A_{m-1}z_{m-1}\quad \text{for} \quad m\le n.
\]
We define $\mathbf{y}=(y_m)_{m\in \Z}$ by $y_m=x_m$ for $m\ge n$ and $y_m=z_m$ for $m<n$. It is easy to verify that $\mathbf{y}\in Y_B$ and $T_B \mathbf{y}=0$. Since $T_B$ is invertible, we have $\mathbf{y}=0$ and thus $y_n=v=0$.
\end{proof}

Let $P_n \colon X \to X(n)$ and $Q_n\colon X \to Z(n)$ be the projections associated to the decomposition in~\eqref{DEC}.

\begin{lemma}
Property~\eqref{*AX} holds.
\end{lemma}

\begin{proof}[Proof of the lemma]
It is sufficient to show that
\[
A_n X(n)\subset X(n+1) \quad \text{and} \quad A_n Z(n)\subset Z(n+1)
\]
for $n\in \Z$. Take $v\in X(n)$ and $\mathbf{x}=(x_m)_{m\in \Z}\in Y_B$ such that $x_n=v$ and
\[
x_m=A_{m-1}x_{m-1} \quad \text{for} \quad m> n.
\]
Then $x_{n+1}=A_n v\in X(n+1)$. Now take $v\in Z(n)$ and choose $\mathbf{z}=(z_m)_{m\in \Z}$ such that $z_n=v$ and $z_m=A_{m-1}z_{m-1}$ for $m\le n$. We define $\mathbf{z'}=(z_m')_{m\in \Z}$ by $z_m'=z_m$ for $m\neq n+1$ and $z_{n+1}=A_n v$. Since $\mathbf{z'}\in Y_B$ and
\[
z_m'=A_{m-1}z_{m-1}' \quad \text{for} \quad m\le n+1,
\]
we conclude that $A_n v\in Z(n+1)$.
\end{proof}

\begin{lemma}
The linear operator $A_n \rvert \ker P_n \colon \ker P_n \to \ker P_{n+1}$ is invertible for each $n\in \Z$.
\end{lemma}

\begin{proof}[Proof of the lemma]
We first establish the injectivity of the operator. Assume that $A_n v=0$ for $v\in \ker P_n=Z(n)$ and choose $\mathbf{z}=(z_m)_{m\in \Z}\in Y_B$ such that $z_n=v$ and
\[
z_m=A_{m-1}z_{m-1} \quad \text{for} \quad m\le n.
\]
Moreover, we define $\mathbf{y}=(y_m)_{m\in \Z}$ by $y_m=0$ for $m>n$ and $y_m=z_m$ for $m\le n$. Clearly, $\mathbf{y}\in Y_B$ and $T_B \mathbf{y}=0$. Since $T_B$ is invertible, we conclude that $\mathbf{y}=0$ and thus $y_n=v=0$.

In order to show that the operator is onto, take $v\in \ker P_{n+1}=Z(n+1)$ and $\mathbf{z}=(z_m)_{m\in \Z}\in Y_B$ with $z_{n+1}=v$ and $z_m=A_{m-1}z_{m-1}$ for $m\le n+1$. Clearly, $z_n \in Z(n)$ and $A_n z_n=z_{n+1}$. This shows that $A_n \rvert \ker P_n$ is onto.
\end{proof}

Now we establish exponential bounds. Take $n\in \Z$ and $v\in X$. Moreover, let $\mathbf{y}$ and $\mathbf{x}$ be as in the proof of Lemma~\ref{L}. For each $z \ge 1$, we define a linear operator
\[
B(z) \colon (\mathcal{D}(T_B),\lVert \cdot \rVert'_{Y_B}) \to (Y_B, \lVert \cdot \rVert_{Y_B})
\]
by
\[
(B(z)\boldsymbol{\nu})_m=
\begin{cases}
z\nu_m-A_{m-1}\nu_{m-1} & \text{if $m \le n$,} \\
\frac{1}{z} \nu_m-A_{m-1}\nu_{m-1} & \text{if $m > n$.}
\end{cases}
\]
We have $B(1)=T_B$ and
\[
\lVert ( B(z)-T_B)\boldsymbol{\nu} \rVert_{Y_B} \le (z-1)\lVert \boldsymbol{\nu}\rVert_{Y_B}'
\]
for $\boldsymbol{\nu} \in \mathcal{D}(T_B)$ and $z\ge 1$. In particular, this implies that $B(z)$ is invertible whenever $1\le z< 1+1/\lVert T_B^{-1}\rVert$, and
\[
\lVert B(z)^{-1}\rVert \le \frac{1}{\lVert T_B^{-1}\rVert^{-1}-(z-1)}.
\]
Take $t=1/z$ for a given $z \in (1,1+1/\lVert T_B^{-1} \rVert)$ and let $\mathbf{z} \in Y_B$ be the unique element such that
$B(1/t)\mathbf{z}=\mathbf{y}$. Writing
\[
D'=\frac{1}{\lVert T_B^{-1}\rVert^{-1}-(1/t-1)},
\]
we obtain
\[
\begin{split}
\lVert \mathbf{z}\rVert_{Y_B} & \le \lVert \mathbf{z}\rVert_{Y_B}' =\lVert B(1/t)^{-1} \mathbf{ y} \rVert_{Y_B}' \\
& \le D'\lVert \mathbf{y}\rVert_{Y_B} =ND'\alpha_B \lVert v\rVert_n
\end{split}
\]
(where $\alpha_B=\lVert \chi_{\{0\}}\rVert_B$). For each $m\in \Z$, let $x_m^*=t^{\lvert m-n\rvert -1}z_m$ and $\mathbf{x^*}=(x_m^*)_{m\in \N}$. Clearly, $\mathbf{x^*}\in Y_B$. One can easily verify that $T_B\mathbf{x^*}=\mathbf{y}$ and hence $\mathbf{x^*}=\mathbf{x}$. Thus,
\begin{equation}\label{*3}
\begin{split}
\lVert x_m\rVert_m & =\lVert x_m^*\rVert_m =t^{\lvert m-n\rvert -1}\lVert z_m\rVert_m \\
& \le \frac{N}{\alpha_B}t^{\lvert m-n\rvert -1} \lVert \mathbf{z}\rVert_{Y_B} \le \frac{N^2D'}{t}t^{\lvert m-n\rvert}\lVert v\rVert_n
\end{split}
\end{equation}
for $m\in \Z$. Moreover, it was shown in the proof of Lemma~\ref{L} that $P_n v=x_n$ and $Q_n v=-A_{n-1}x_{n-1}$. Hence, it follows from~\eqref{a1} and~\eqref{*3} that
\begin{equation}\label{AA}
\begin{split}
\lVert \cA(m,n)P_nv\rVert_m &=\lVert \cA(m,n)x_n\rVert_m =\lVert x_m\rVert_m \\
&\le \frac{N^2D'}{t}t^{m-n}\lVert v\rVert_n
\end{split}
\end{equation}
for $m\ge n$. Similarly, it follows from~\eqref{a1} and~\eqref{*3} that
\begin{equation}\label{EF}
\lVert \cA(m,n)Q_n v\rVert_m \le \frac{N^2D'}{t}t^{n-m}\lVert v\rVert_n
\end{equation}
for $m<n$. By \eqref{AA} and~\eqref{EF}, there exists $D>0$ such that~\eqref{eq:1} and~\eqref{eq:2} hold taking $\lambda=t$ and $\mu=1/t$. This completes the proof of the theorem.
\end{proof}

\section{Nonuniformly hyperbolic sets}

Now we consider an elaboration of the situation considered in Section~\ref{S3}. Namely, we characterize the notion of a nonuniformly hyperbolic set in terms of the invertibility of certain linear operators. More precisely, to each trajectory $f^n(x)$ of a nonuniformly hyperbolic set of a diffeomorphism~$f$ one can associate a linear operator defined in terms of the sequence of tangent spaces $d_{f^n(x)} f$ (see the discussion after Definition~\ref{DEF}). Moreover, each trajectory admits an exponential dichotomy with respect to the same sequence of tangent spaces and so it is natural to use arguments that are an elaboration of those in the former section.

\subsection{Basic notions}

Let $M$ be a compact Riemannian manifold and let $f\colon M \to M$ be a $C^{1}$ diffeomorphism.

\begin{definition}
An $f$-invariant measurable set $\Lambda \subset M$ is said to be \emph{nonuniformly hyperbolic} if there exist constants $0 < \lambda < 1 < \mu$ and a $df$-invariant splitting
\[
T_x M= E^s(x) \oplus E^u(x)
\]
for $x \in \Lambda$ such that given $\epsilon >0$, there exist measurable functions $C,K\colon \Lambda \to \R^+$ such that for each $x\in \Lambda$:
\begin{enumerate}
\item for $v\in E^s(x)$ and $n\ge 0$,
\begin{equation}\label{1}
\lVert d_x f^n v \rVert_{f^n(x)} \le C(x)\lambda^n e^{\epsilon n} \lVert v \rVert_x;
\end{equation}
\item for $v\in E^u(x)$ and $n\ge 0$,
\begin{equation}\label{2}
\lVert d_x f^{-n}v \rVert_{f^{-n}(x)} \le C(x)\mu^{-n}e^{\epsilon n} \lVert v\rVert_x;
\end{equation}
\item
\begin{equation}\label{3}
\angle (E^s(x),E^u(x))\ge K(x);
\end{equation}
\item for $n\in \Z$,
\begin{equation}\label{7}
C(f^n(x))\le C(x)e^{\epsilon \lvert n \rvert} \quad \text{and} \quad K(f^n(x))\ge K(x)e^{-\epsilon \lvert n \rvert}.
\end{equation}
\end{enumerate}
\end{definition}

We note that a nonuniform hyperbolic set gives rise naturally to a parameterized family of exponential dichotomies with respect to a sequence of norms. More precisely, to each trajectory one can associate an exponential dichotomy (see~\cite{BP2}).

\begin{proposition}\label{prp.6}
Let $\Lambda \subset M$ be a nonuniformly hyperbolic set. Then for each $\epsilon > 0$ such that $\lambda e^\epsilon < 1 < \mu e^{-\epsilon} $ there exists a norm $\lVert \cdot \rVert ' =\lVert \cdot \rVert^\epsilon$ on $T_\Lambda M$ such that for each $x \in \Lambda$ the sequence of linear operators
\[
A_n =d_{f^n(x)} f \colon T_{f^n(x)} M \to T_{f^{n+1}(x)} M
\]
admits an exponential dichotomy with respect to the norms $\lVert \cdot \rVert'_{f^n(x)}$.
\end{proposition}

Alternatively, Proposition~\ref{prp.6} can be obtained as a consequence of the proof of Theorem~\ref{thm.5} below (the proof introduces a particular norm that is also adapted to our characterization of nonuniformly hyperbolic sets).

\subsection{Characterization of nonuniformly hyperbolic sets}

Given an admissible Banach sequence space $B$ and a norm $\lVert \cdot \rVert'$ on the tangent bundle $T_{\Lambda}M$, for each $x \in \Lambda$ we denote by $Y_x$ the set of all sequences $\boldsymbol{\mu}=(\mu_n)_{n\in \Z}$ with $\mu_n \in T_{x_n}M$, where $x_n=f^n(x)$, such that $(\lVert \mu_n \rVert_{x_n}')_{n\in \Z} \in B$. One can easily verify that $Y_{x}$ is a Banach space with the norm
\[
\lVert \boldsymbol{\mu}\rVert= \lVert (\lVert \mu_n \rVert_{x_n}')_{n\in \Z}\rVert_B.
\]
Finally, we define a linear operator $R_x$ by
\[
(R_x\boldsymbol{\mu})_n =\mu_n -d_{x_{n-1}} f \mu_{n-1}, \quad n\in \Z,
\]
on the domain formed by all $\boldsymbol{\mu}=(\mu_n)_{n\in \Z} \in Y_x$ such that $R_x\boldsymbol{\mu} \in Y_x$.

\begin{theorem}\label{thm.5}
Let $\Lambda \subset M$ be a nonuniformly hyperbolic set and let $B$ be an admissible Banach sequence space. Then there exists $\epsilon_0 >0$ such that for every $\epsilon \in (0, \epsilon_0)$ there is a norm $\lVert \cdot \rVert' = \lVert \cdot \rVert^ \epsilon $ on $T_{\Lambda} M$ and a measurable function $G\colon \Lambda \to \R^+$ such that for each $x \in \Lambda$:
\begin{enumerate}
\item \begin{equation}\label{G1}
\frac{1}{2}\lVert v\rVert_x \le \lVert v\rVert_x^{\epsilon} \le G(x)\rVert v\rVert_x, \quad v\in T_x M;
\end{equation}
\item \begin{equation}\label{G2}
G(f^n(x)) \le e^{2\epsilon \lvert n\rvert}G(x), \quad n\in \Z;
\end{equation}
\item
$R_x \colon Y_x \to Y_x$ is a well defined, bounded and invertible linear operator;
\item
there exists a constant $D>0$ (independent of $\epsilon$ and $x$) such that
\begin{equation}\label{BI}
\lVert R_x^{-1} \rVert \le D.
\end{equation}
\end{enumerate}
\end{theorem}

\begin{proof}
Since $M$ is compact and $f$ is continuous, there exists $A>0$ such that $\lVert d_x f\rVert \le A$ and $\lVert d_x f^{-1}\rVert \le A$ for $x\in M$. Without loss of generality, one may assume that $1/ A \le \lambda$ and $\mu \le A$ (since otherwise one can simply increase~$A$). Take $\epsilon_0>0$ such that $\lambda e^{\epsilon_0} < 1 < \mu e^{-\epsilon_0}$. For each $\epsilon \in (0,\epsilon_0)$ we introduce an adapted norm $\lVert \cdot \rVert^{\epsilon}$ on $T_{\Lambda}M$. For $v\in E^s(x)$, let
\[
\lVert v\rVert_x^{\epsilon}= \sup_{n\ge 0}\big(\lambda^{-n}e^{-\epsilon n}\lVert d_x f^n v\rVert_{f^n(x)} \big)+ \sup_{n <0}\big( e^{\epsilon n}A^n\lVert d_x f^n \rVert_{f^n(x)}\big).
\]
It follows from~\eqref{1} that
\begin{equation}\label{B1}
\lVert v\rVert_x \le \lVert v\rVert_x^{\epsilon} \le (C(x)+1)\lVert v\rVert_x \quad \text{for} \quad v\in E^s(x).
\end{equation}
Moreover,
\begin{equation}\label{T1}
\begin{split}
\lVert d_x f v\rVert_{f(x)}^{\epsilon} &=\sup_{n\ge 0}\big(\lambda^{-n}e^{-\epsilon n}\lVert d_x f^{n+1} v\rVert_{f^{n+1}(x)} \big) \\
&\phantom{=}+ \sup_{n <0} \big(e^{\epsilon n} A^n\lVert d_x f^{n+1}v\rVert_{f^{n+1}(x)} \big) \\
&= \lambda e^{\epsilon }\sup_{n\ge 0} \big(\lambda^{-(n+1)}e^{-\epsilon (n+1)}\lVert d_x f^{n+1} v\rVert_{f^{n+1}(x)} \big) \\
&\phantom{=}+
\frac{1}{A}e^{-\epsilon }\sup_{n <0}\big(A^{n+1}e^{\epsilon (n+1)}\lVert d_x f^{n+1} v\rVert_{f^{n+1}(x)} \big) \\
& \le \lambda e^{\epsilon }\lVert v\rVert_x^{\epsilon}
\end{split}
\end{equation}
for $v\in E^s(x)$. Similarly, for $v\in E^u(x)$, let
\[
\lVert v\rVert_x^{\epsilon}=\sup_{n \le 0}\big( \mu^{-n}e^{\epsilon n}\lVert d_xf^nv\rVert_{f^{n}(x)}\big)+\sup_{n > 0}\big(A^{-n}e^{-\epsilon n}\lVert d_xf^n v\rVert_{f^{n}(x)} \big).
\]
It follows from~\eqref{2} that
\begin{equation}\label{B2}
\lVert v\rVert_x \le \lVert v\rVert_x^{\epsilon} \le (C(x)+1)\lVert v\rVert_x \quad \text{for} \quad v\in E^u(x).
\end{equation}
Moreover,
\begin{equation}\label{T3}
\begin{split}
\lVert d_xf^{-1}v\rVert_{f^{-1}(x)}^{\epsilon} &=\sup_{n \le 0}\big(\mu^{-n}e^{\epsilon n}\lVert d_x f^{n-1} v\rVert_{f^{n-1}(x)} \big)\\
&\phantom{=}+\sup_{n > 0}\big( e^{-\epsilon n} A^{-n}\lVert d_xf^{n-1}v\rVert_{f^{n-1}(x)}\big) \\
&=\frac{1}{\mu}e^{\epsilon }\sup_{n \le 0}\big(\mu^{-(n-1)}e^{\epsilon (n-1)}\lVert d_xf^{n-1} v\rVert_{f^{n-1}(x)} \big) \\
&\phantom{=}+\frac{1}{A}e^{-\epsilon}\sup_{n > 0}\big(e^{-\epsilon (n-1)}A^{-(n-1)}\lVert d_xf^{n-1} v\rVert_{f^{n-1}(x)} \big) \\
& \le \frac{1}{\mu}e^{\epsilon }\lVert v\rVert_x^{\epsilon}
\end{split}
\end{equation}
for $v\in E^u(x)$. One can show in a similar manner that
\begin{equation}\label{T2}
\lVert d_xf v\rVert_{f(x)}^{\epsilon} \le A(e^{\epsilon}+1) \lVert v\rVert_x^{\epsilon} \quad \text{for} \quad v\in E^u(x).
\end{equation}
For an arbitrary $v\in T_x M$, we define
\[
\lVert v\rVert_x^{\epsilon} =\max \big\{ \lVert v^s\rVert_x^{\epsilon}, \lVert v^u \rVert_x^{\epsilon} \big\},
\]
where $v=v^s+v^u$ with $v^s \in E^s(x)$ and $v^u \in E^u(x)$. It follows from~\eqref{3}, \eqref{B1} and~\eqref{B2} that
\[
\frac{1}{2}\lVert v\rVert_x \le \lVert v\rVert_x^{\epsilon} \le \frac{C(x)+1}{K(x)}\lVert v\rVert_x \quad \text{for} \quad v\in T_x M.
\]
Hence, \eqref{G1} holds taking $G(x)=(C(x)+1)/K(x)$. Moreover, it follows from~\eqref{7} that~\eqref{G2} holds. Finally, it follows from~\eqref{T1} and~\eqref{T2} that
\begin{equation}\label{T4}
\lVert d_xf v\rVert_{f(x)}^{\epsilon} \le A(e^{\epsilon}+1) \lVert v\rVert_x^{\epsilon}
\end{equation}
for $x\in \Lambda$ and $v\in T_x M$.

Now let $P(x) \colon T_x M \to E^s(x)$ and $Q(x) \colon T_x M \to E^u(x)$ be the projections associated to the decomposition $T_x M=E^s(x) \oplus E^u(x)$.
\begin{lemma}\label{AUX}
There exists a constant $Z>0$ (independent of $\epsilon $ and $x$) such that
\begin{equation}\label{BB}
\lVert P(x) v\rVert_x^{\epsilon} \le Z\lVert v\rVert_x^{\epsilon} \quad \text{and} \quad \lVert Q(x) v\rVert_x^{\epsilon} \le Z\lVert v\rVert_x^{\epsilon}
\end{equation}
for $x\in \Lambda$ and $v\in T_x M$.
\end{lemma}

\begin{proof}[Proof of the lemma]
For each $x \in \Lambda$ let
\[
\gamma_x^{\epsilon}=\inf \big\{\lVert v^s+v^u\rVert_x^{\epsilon}: \lVert v^s\rVert_x^{\epsilon}=\lVert v^u\rVert_x^{\epsilon}=1, v^s\in E^s(x), v^u \in E^u(x) \big\}.
\]
Take a vector $v\in T_x M$ such that $P v\neq 0$ and $Q v\neq 0$, where $P=P(x)$ and $Q=Q(x)$. Then
\[
\begin{split}
\gamma_x^{\epsilon} \le \bigg{\lVert} \frac{P v}{\lVert P v\rVert_x^{\epsilon}}+ \frac{Q v}{\lVert Q v\rVert_x^{\epsilon}} \bigg{\rVert}_x^{\epsilon} &=\frac{1}{\lVert P v\rVert_x^{\epsilon}}\bigg{\lVert} P v+\frac{\lVert P v\rVert_x^{\epsilon}}{\lVert Qv\rVert_x^{\epsilon}}Q v\bigg{\rVert}_x^{\epsilon} \\
&= \frac{1}{\lVert P v\rVert_x^{\epsilon}}\bigg{\lVert} v+\frac{\lVert P v\rVert_x^{\epsilon}-\lVert Qv\rVert_x^{\epsilon}}{\lVert Qv\rVert_x^{\epsilon}}Q v\bigg{\rVert}_x^{\epsilon} \\
& \le \frac{2\lVert v\rVert_x^{\epsilon}}{\lVert Pv\rVert_x^{\epsilon}}
\end{split}
\]
and thus,
\[
\lVert Pv\rVert_x^{\epsilon} \le \frac{2}{\gamma_x^{\epsilon}}\lVert v\rVert_x^{\epsilon}
\]
for $v\in T_x M$. In order to estimate $\gamma_x^{\epsilon}$, take $v^s\in E^s(x), v^u \in E^u(x) $ such that $ \lVert v^s\rVert_x^{\epsilon}=\lVert v^u\rVert_x^{\epsilon}=1$. It follows from~\eqref{T1}, \eqref{T3} and~\eqref{T4} (recall that $\epsilon <\epsilon_0$) that
\[
\begin{split}
\lVert v^s+v^u \rVert_x^{\epsilon}& \ge \frac{1}{A(e^{\epsilon_0}+1)}\lVert d_x f(v^s+v^u)\rVert_{f(x)}^{\epsilon} \\
& \ge \frac{1}{A(e^{\epsilon_0}+1)} \big(\lVert d_xfv^u \rVert_{f(x)}^{\epsilon}-\lVert d_xfv^s \rVert_{f(x)}^{\epsilon}\big) \\
& \ge \frac{1}{A(e^{\epsilon_0}+1)}(\mu e^{-\epsilon_0}-\lambda e^{\epsilon_0})
\end{split}
\]
and thus,
\[
\gamma_x^{\epsilon} \ge \frac{1}{A(e^{\epsilon_0}+1)}(\mu e^{-\epsilon_0}-\lambda e^{\epsilon_0}).
\]
Therefore, \eqref{BB} holds taking
\[
Z=\frac{2A(e^{\epsilon_0}+1)}{\mu e^{-\epsilon_0}-\lambda e^{\epsilon_0}}.
\]
This completes the proof of the lemma.
\end{proof}

Now take $x\in \Lambda$. It follows from~\eqref{T4} that $R_x$ is a well defined bounded linear operator on $Y_x$. We first show that it is onto. Let $\boldsymbol{\mu}=(\mu_n)_{n\in \Z} \in Y_x$. By Lemma~\ref{AUX}, we have $\boldsymbol{\mu^s}=(\mu_n^s)_{n\in \Z} \in Y_x$ and $\boldsymbol{\mu^u}=(\mu_n^u)_{n\in \Z} \in Y_x$, where
\[
\mu_n^s = P(f^n(x)) \mu_n\quad\text{and}\quad \mu_n^u =Q(f^n(x)) \mu_n.
\]
For each $n\in \Z$, let
\[
\xi _{n}^{s} =\sum_{m\geq 0}d_{x_{n-m}} f^{m} \mu _{n-m}^{s}
\]
and
\[
\xi _{n}^{u} =-\sum_{m \geq 1}d_{x_{n+m}} f^{-m} \mu _{n+m}^{u}.
\]
It follows from~\eqref{bnd}, \eqref{T1}, \eqref{T3} and~\eqref{BB} (since $\epsilon <\epsilon_0$) that $\boldsymbol{\xi^s}=(\xi_n^s)_{n\in \Z}$ and $\boldsymbol{\xi^u}=(\xi_n^u)_{n\in \Z}$ belong to $Y_x$. Moreover,
\[
\lVert \boldsymbol{\xi^s} \rVert \le \frac{1}{1-\lambda e^{\epsilon_0}}Z \lVert \boldsymbol{\mu}\rVert \quad \text{and} \quad \lVert \boldsymbol{\xi^u} \rVert \le \frac{1}{\mu e^{-\epsilon_0}-1} Z\lVert \boldsymbol{\mu}\rVert
\]
for $n\in \Z$. Therefore, $\boldsymbol{\xi}=(\xi_n)_{n\in \Z}$, where $\xi_n=\xi_n^s+\xi_n^u$, belongs to $Y_x$ and
\begin{equation}\label{N}
\lVert \boldsymbol{\xi}\rVert \le Z\bigg{(}\frac{1}{1-\lambda e^{\epsilon_0}}+\frac{1}{\mu e^{-\epsilon_0}-1} \bigg{)}\lVert\boldsymbol{\mu}\rVert.
\end{equation}
Moreover, one can easily verify that $R_x\boldsymbol{\xi}=\boldsymbol{\mu}$.

Now we show that $R_x$ is injective. Assume that $R_x \boldsymbol{\xi}=0$ for some $\boldsymbol{\xi}=(\xi_n)_{n\in \Z} \in Y_x$. Then
$\xi_n=d_{x_{n-1}} f$ for $n\in \Z$ and hence, $\xi_n^s=d_{x_{n-1}} f \xi_{n-1}^s$ and $\xi_n^u=d_{x_{n-1}} f\xi_{n-1}^u$ for $n\in \Z$. For each $k\in \Z$, it follows from~\eqref{T1} that
\[
\lVert \xi_k^s \rVert_{x_k}^{\epsilon} \le (\lambda e^{\epsilon })^m \lVert \xi_{k-m}^s\rVert_{x_{k-m}}^\epsilon \le \frac{NZ}{\alpha_B} (\lambda e^{\epsilon_0})^m \lVert \boldsymbol{\xi}\rVert
\]
for $m \ge 0$. Letting $m\to +\infty$, since $\lambda e^{\epsilon_0} <1$ we obtain $\xi_k^s =0$. Similarly, $\xi_k^u=0$ for $k\in \Z$ and thus $\boldsymbol{\xi}=0$.
This shows that $R_x$ is invertible. In addition, it follows from~\eqref{N} that there exists a constant $D>0$ (independent on $x$ and $\epsilon$) such that~\eqref{BI} holds. This completes the proof of the theorem.
\end{proof}

Now we establish the converse of Theorem~\ref{thm.5}.

\begin{theorem}
Let $\Lambda \subset M$ be an $f$-invariant measurable set and let $B$ be an admissible Banach sequence space. Assume that there exist $D>0$ and $\epsilon_0 >0$ such that for each $\epsilon \in (0,\epsilon_0)$ there is a norm $\lVert \cdot \rVert^{\epsilon}$ on $T_{\Lambda}M$ and a measurable function $G \colon \Lambda \to \R^+$ such that for each $x \in \Lambda$:
\begin{enumerate}
\item \eqref{G1} and~\eqref{G2} hold;
\item $R_x \colon Y_x \to Y_x$ is a well defined bounded invertible linear operator and~\eqref{BI} holds.
\end{enumerate}
Then $\Lambda$ is a nonuniformly hyperbolic set.
\end{theorem}

\begin{proof}
Take $x\in \Lambda$ and $v\in T_xM$. We define $\boldsymbol{\mu}=(\mu_n)_{n\in \Z}$ by $\mu_0=v$ and $\mu_n =0$ for $n \neq 0$. Clearly, $\boldsymbol{\mu} \in Y_x$. Now take $\boldsymbol{\xi}=(\xi_n)_{n\in \Z} \in Y_x$ such that $R_x \boldsymbol{\xi}=\boldsymbol{\mu}$. It can be written in the form
\[
\xi_n=\begin{cases}
d_{x_{n-1}} f\xi_{n-1}, & n \ne 0, \\
d_{x_{-1}} f_{-1} + v, & n=0.
\end{cases}
\]
We will show that $v=v^s+v^u$, where $v^s=\xi_0$ and $v^u=-d_{x_{-1}} f \xi_{-1}$ is the hyperbolic splitting. For each $z \ge 1$ we define an operator $B(z)$ on $Y_x$ by
\[
(B(z)\boldsymbol{\nu})_m=
\begin{cases}
z\nu_m-d_{x_{m-1}} f \nu_{m-1} & \text{if $m \le 0$,} \\
\frac{1}{z} \nu_m-d_{x_{m-1}}f\nu_{m-1}\nu_{m-1} & \text{if $m > 0$.}
\end{cases}
\]
 Clearly,
\[
\lVert (B(z)-R_x)\boldsymbol{\nu}\rVert \le (z-1)\lVert \boldsymbol{\nu}\rVert
\]
for $\boldsymbol{\nu} \in Y_x$ and $z\ge 1$. Therefore, $B(z)$ is invertible whenever $1 \le z < 1+ 1/D$, and
\[
\lVert B(z)^{-1}\rVert \le \frac{1}{D^{-1}-(z-1)}.
\]
Now take $\lambda \in (0,1)$ (independently on $\epsilon$) such that $\lambda^{-1} <1+1/D$ and take $\boldsymbol{\xi^*}\in Y_x$ such that $B(\lambda^{-1})\boldsymbol{\xi^*}= \boldsymbol{\mu}$. Writing
\[
D'=\frac{1}{D^{-1}-(\lambda^{-1}-1)},
\]
we obtain
\[
\lVert \boldsymbol{\xi^*}\rVert=\lVert B(\lambda^{-1})^{-1}\boldsymbol{\mu}\rVert \le D'\lVert \boldsymbol{\mu}\rVert =D'\alpha_B \lVert v\rVert_x^{\epsilon}.
\]
For each $m\in \Z$, let $\overline \xi_m=\lambda^{\lvert m\rvert-1}\xi_m^*$ and $\overline{\boldsymbol{\xi}}=(\overline \xi_m)_{m\in \Z}$. Clearly, $\overline{\boldsymbol{\xi}}\in Y_x$. Moreover, one can easily verify that
$R_{\mathbf{x}}\overline{\boldsymbol{\xi}}=\boldsymbol{\mu}$ and hence $\overline{\boldsymbol{\xi}}=\boldsymbol{\xi}$. Thus,
\[
\lVert \xi_m \rVert_{x_m}^{\epsilon}=\lVert \overline \xi_m \rVert_{x_m}^{\epsilon}=\lambda^{\lvert m\rvert -1}\lVert \xi_m^*\rVert_{x_m}^{\epsilon} \le D'N\lambda^{\lvert m\rvert -1} \lVert v\rVert_x^{\epsilon}
\]
for $m\in \Z$. Finally, it follows from~\eqref{G1} that
\[
\lVert d_xf^m v^s\rVert \le C(x)\lambda^m \lVert v\rVert \quad \text{and} \quad \lVert d_x f^{-m} v^u\rVert \le C(x)\lambda^{m}\lVert v\rVert
\]
for $m \ge 1$, where $C(x)=2D'NG(x)\lambda^{-1}$.

Now let $E^s(x)$ and $E^u(x)$ be the sets, respectively, of all vectors $v^s$ and~$v^u$ constructed above. These are $df$-invariant subspaces of $T_x M$ and are uniquely defined (and so independent of~$\epsilon$). Indeed, take $v\in T_x M$ and let $v=v^s +v^u$ with $v^s\in E^s(x)$ and $v^u\in E^u(x)$. We define $\boldsymbol{\mu}=(\mu_n)_{n\in \Z}$ by $\mu_0=d_x f v$ and $\mu_n=0$ for $n\neq 0$. Clearly, $\boldsymbol{\mu} \in Y_{f(x)}$. Moreover, we define $\boldsymbol{\xi}=(\xi_n)_{n\in \Z}$ by $\xi_n=d_xf^{n+1}v^s$ for $n \ge 0$ and $\xi_n=-d_xf^{n+1}v^u$ for $n< 0$. Then $\boldsymbol{\xi} \in Y_{f(x)}$ (this is a consequence of the fact that the sequence $\boldsymbol{\xi}$ constructed above belongs to~$Y_x$). Finally, it is easy to check that $R_{f(x)}\boldsymbol{\xi}=\boldsymbol{\mu}$. This implies that
\[
\xi_0=d_x fv^s\in E^s(f(x))\quad\text{and}\quad -d_xf\xi_{-1}=d_x fv^u\in E^u(f(x))
\]
is the hyperbolic splitting of $d_x fv$ and so the decomposition is $df$-invariant.

We have
\[
\lVert v^s\rVert_x^{\epsilon}=\lVert \xi_0 \rVert_x^{\epsilon} \le \frac{1}{\alpha_B} \lVert \boldsymbol{\xi}\rVert \le \frac{D}{\alpha_B} \lVert \boldsymbol{\mu}\rVert= D\lVert v\rVert_x^{\epsilon}
\]
and thus,
\[
\lVert v^u\rVert_x^{\epsilon}=\lVert v-v^s\rVert_x^{\epsilon} \le \lVert v\rVert_x^{\epsilon} +\lVert v^s\rVert_x^{\epsilon} \le (1+D)\lVert v\rVert_x^{\epsilon}.
\]
By~\eqref{G1}, we obtain
\[
\lVert v^s\rVert_x \le \frac{1}{K(x)}\lVert v\rVert_x \quad \text{and} \quad \lVert v^u\rVert_x \le \frac{1}{K(x)}\lVert v\rVert_x,
\]
where $K(x)=1/((2+2D)G(x))$. It follows readily from~\eqref{G2} that the functions $C$ and $K$ satisfy~\eqref{7} with $\epsilon$ replaced by $2\epsilon$. This shows that the set $\Lambda$ is nonuniformly hyperbolic.
\end{proof}

\bibliographystyle{amsplain}

\begin{thebibliography}{11}
\bibitem{BDV}
L. Barreira, D. Dragi\v cevi\' c and C. Valls, \emph{Exponential dichotomies with respect to a sequence of norms and admissibility}, Int. J. Math. \textbf{25} (2014), 1450024, 20 pp.

\bibitem{BP2}
L. Barreira and Ya. Pesin, \emph{Nonuniform Hyperbolicity}, Encyclopedia of Mathematics and Its Application 115, Cambridge University Press, 2007.

\bibitem{CLb}
C. Chicone and Yu.~Latushkin, \emph{Evolution Semigroups in
Dynamical Systems and Differential Equations}, Mathematical Surveys
and Monographs 70, Amer. Math. Soc., 1999.

\bibitem{DK}
Ju. Dalec$'$ki\u{\i} and M. Kre{\u\i}n, \emph{Stability of Solutions
of Differential Equations in Banach Space}, Translations of
Mathematical Monographs 43, Amer. Math. Soc., 1974.

\bibitem{DS}
D. Dragi\v{c}evi\'{c} and S. Slijep\v{c}evi\'{c}, \emph{Characterization of hyperbolicity and generalized shadowing lemma},
Dyn. Syst. \textbf{26} (2011), 483--502.

\bibitem{H}
J. Hale, \emph{Asymptotic Behavior of Dissipative Systems},
Mathematical Surveys and Monographs 25, Amer. Math. Soc., 1988.

\bibitem{He}
D. Henry, \emph{Geometric Theory of Semilinear Parabolic Equations},
Lect. Notes in Math. 840, Springer, 1981.

\bibitem{HH}
N. Huy, \emph{Exponential dichotomy of evolution equations and admissibility of function spaces on a half-line},
J. Funct. Anal. \textbf{235} (2006), 330--354.

\bibitem{LV}
B. Levitan and V. Zhikov, \emph{Almost Periodic Functions and Differential Equations}, Cambridge University Press, 1982.

\bibitem{MS}
J. Massera and J. Sch\"affer, \emph{Linear Differential Equations
and Function Spaces}, Pure and Applied Mathematics 21, Academic
Press, 1966.

\bibitem{M} J. Mather, \emph{Characterization of Anosov
diffeomorphisms}, Indag. Math. \textbf{30} (1968), 479--483.

\bibitem{SY}
G. Sell and Y. You, \emph{Dynamics of Evolutionary Equations},
Applied Mathematical Sciences 143, Springer, 2002.

\end{thebibliography}

\end{document}